\documentclass[final,12pt,3p]{elsarticle}

\usepackage{graphicx}
\usepackage{amssymb}
\usepackage{amsthm}
\usepackage{amsmath}
\usepackage{algorithm}
\usepackage{algorithmic}
\usepackage{longtable} 
\usepackage{tikz}
\usepackage{lineno}
\usepackage{verbatim} 
\usepackage{hyperref}
\usepackage{xcolor}
\usepackage{caption}
\usetikzlibrary{shapes}
\usepackage{subcaption}
\usepackage{mathrsfs} 
\usepackage{array}
\usepackage{multirow}
\usepackage{booktabs} 
\usepackage{adjustbox}
\usepackage[title]{appendix}

\biboptions{sort&compress}
\hypersetup{
colorlinks=true,
linkcolor=red,
filecolor=green,
urlcolor=green,
anchorcolor=green,
citecolor=cyan,
pdftitle={Overleaf Example},
pdfpagemode=FullScreen,
}
\theoremstyle{plain}
\newtheorem{theorem}{Theorem}
\newtheorem{proposition}[theorem]{Proposition}

\theoremstyle{definition}

\theoremstyle{remark}
\newtheorem{remark}{Remark}

\begin{document}

\begin{frontmatter}

\title{Structure-Preserving Neural ODEs via Nonstandard Finite Difference Discretization}

\author[BUW,MAIS]{Achraf Zinihi}
\ead{a.zinihi@edu.umi.ac.ma} 

\author[BUW]{Matthias Ehrhardt\corref{Corr}}
\cortext[Corr]{Corresponding author}
\ead{ehrhardt@uni-wuppertal.de}

\author[MAIS]{Moulay Rchid Sidi Ammi}
\ead{rachidsidiammi@yahoo.fr}

\address[BUW]{University of Wuppertal, 
Applied and Computational Mathematics,\\
Gaußstrasse 20, 42119 Wuppertal, Germany}

\address[MAIS]{Moulay Ismail University of Meknes, FST Errachidia,\\ Department of Mathematics, AMNEA Group,
Errachidia 52000, Morocco}


\begin{abstract}
Although neural ordinary differential equations (NODEs) are a powerful framework for learning continuous-time dynamics, they generally do not preserve essential qualitative properties, such as positivity. 
We propose a structure-preserving Neural ODE framework based on nonstandard finite difference (NSFD) discretization. 
The learned dynamics are parameterized by nonnegative production and destruction rates, yielding an explicit, differentiable update that integrates seamlessly into standard automatic differentiation pipelines. We prove that the resulting scheme unconditionally preserves positivity for arbitrary time-step sizes while retaining first-order consistency. 
We outline an extension based on Patankar-type discretizations that preserves conservation laws exactly. 
Numerical experiments on an SIR epidemic model show that our approach generates physically meaningful trajectories, remains robust under coarse discretizations, and outperforms conventional NODEs in preserving the qualitative structure of the learned dynamics.

\end{abstract}


\journal{Applied Mathematics Letters}

\end{frontmatter}


\section{Introduction}\label{S1}
Neural ODEs \citep{chen2018neural} parameterize the right-hand side of an initial value problem $\dot{x} = f_\theta(x,t)$ by a neural network $f_\theta$ and treat the network as an implicit, continuous-depth model that is fit to data by differentiating through a numerical ODE solver. 
This formulation is advantageous for learning the latent dynamics of physical and biological systems, including epidemic models, because it preserves the continuous-time structure of the underlying phenomenon rather than treating time as a discrete, recurrent index. 
However, in these applications, the state variables usually have a physical meaning, such as compartment sizes, concentrations, or population counts, and they must remain nonnegative. Often, they also satisfy additional invariants, such as a constant total population. 
A generic $f_\theta$ provides no such guarantee: nothing in the training objective prevents the learned flow from producing negative states, particularly under extrapolation or coarse time stepping, and the standard remedies (clipping, penalty terms for constraint violations, or projection back onto the feasible region after each solver step) are post-hoc corrections rather than structural guarantees. Furthermore, these remedies can also introduce non-smoothness or bias into the learned dynamics.

A separate line of work in numerical analysis, dating back to \cite{mickens1994nonstandard} and rigorously developed  by \cite{anguelov2001contributions} and many others (see \cite{Zinihi2026NSFDpLap} and the references therein for recent high-order constructions), shows that \textit{nonstandard finite difference} (NSFD) schemes can be designed so that the discrete scheme preserves the positivity, boundedness, and other qualitative features of a continuous dynamical system for \textit{every} step size, not just for a step size below some stability threshold. 
The central device is a nonlocal treatment of the right-hand side, along with a generalized, model-dependent denominator function that replaces $\Delta t$. 
This approach has been applied extensively to compartmental epidemic models with prescribed (not learned) right-hand sides \cite{Zinihi2025NSFD}. 
In the NODE literature, "structure" usually refers to a soft, statistical constraint, such as an information bottleneck for dimensionality reduction 
or a linear-in-features form motivated by Koopman operator theory, 
rather than a hard, provable guarantee on the state itself. 
By contrast, this work 
provides a structural constraint of the latter, provable kind.

This work contributes by combining two approaches: we constrain a NODE's vector field to a gain/loss (production/destruction) form and pair it with an NSFD time-stepping rule. 
This guarantees the positivity of the learned dynamics unconditionally, rather than approximately through penalization.
Specifically, we: 
(i) parameterize the network so that it outputs nonnegative gain and loss rates rather than an unconstrained vector field (Section~\ref{S2});
(ii) construct an NSFD update for this class of networks and prove unconditional positivity for any step size (Section~\ref{S3});
(iii) show that the update is a closed-form, differentiable function of the network output; therefore, it requires no additional computation beyond that used in a standard NODE training loop (Proposition~\ref{P2});
(iv) outline how a total-population invariant can be preserved exactly by transitioning from an independent gain/loss parameterization to a flux-based, Patankar-type scheme  \cite{Burchard2003} (Section~\ref{S4}); and
(v) Demonstrate the method on a latent SIR learning problem by comparing it to an unconstrained NODE baseline (Section~\ref{S5}).

\section{Gain-Loss Neural ODEs}\label{S2}
Let $x(t) \in \mathbb{R}_+^d$ denote the latent state, and suppose the true generative dynamics admit a decomposition into nonnegative gain and loss terms,
\begin{equation}\label{E1}
    \dot{x}_i(t) = G_i(x(t),t) - L_i(x(t),t)\,x_i(t), \quad i = 1,\dots,d,
\end{equation}
with $G_i \ge 0$ and $L_i \ge 0$ are sufficiently smooth functions on $\mathbb{R}_+^d$. 
This is not a restriction on the class of dynamics: any locally Lipschitz vector field $f$ satisfying the subtangential condition $f_i(x) \ge0$ whenever $x_i = 0$ (the standard necessary and sufficient condition for $\mathbb{R}_+^d$ to be positively invariant) can be written in this form, e.g. by taking $L_i \equiv 0$ and $G_i = f_i$ wherever $f_i \ge 0$, and folding negative contributions into $L_i x_i$ elsewhere; the decomposition simply makes the sign structure explicit rather than leaving it implicit in a single network. 
This decomposition is used to define the positive and negative parts of $f_\theta$. For further details, we refer the reader to \cite{Zinihi2026NSFDpLap}.
Compartmental epidemic models are a natural fit: 
in the SIR model of Section~\ref{S5}, new infections are a gain for $I$ and a loss for $S$, and recovery is a gain for $R$ and a loss for $I$.

We parameterize $G_i$ and $L_i$ by feed-forward networks with a final softplus (or squared) output activation
\begin{equation*}
    G_{\theta,i}(x,t) = \operatorname{softplus}\bigl(g_{\theta,i}(x,t)\bigr) \quad\text{ and }\quad 
    L_{\theta,i}(x,t) = \operatorname{softplus}\bigl(\ell_{\theta,i}(x,t)\bigr),
\end{equation*}
where $g_{\theta,i}, \ell_{\theta,i}\colon\mathbb{R}^d \times [0,T]\to\mathbb{R}$ are unconstrained networks. 
This guarantees $G_{\theta,i} \ge 0$ and $L_{\theta,i} \ge 0$ for every $\theta$, which is the only hypothesis that Theorem~\ref{T1} below requires. No constraint is imposed on $\theta$ during training.

\section{An NSFD Update for Gain-Loss Neural ODEs}\label{S3}
Given a time grid $t_n = n\Delta t$ and state $x^n \approx x(t_n)$, we propose the following update
\begin{equation}\label{E2}
        \frac{x_i^{n+1} - x_i^n}{\phi(\Delta t)} = G_{\theta,i}(x^n,t_n) - L_{\theta,i}(x^n,t_n)\, x_i^{n+1}, \quad i = 1,\dots,d,
\end{equation}
where the gain term is evaluated explicitly at $x^n$, and the loss term is evaluated semi-implicitly, by multiplying the \textit{new} state $x_i^{n+1}$. 
The denominator function $\phi$ replaces $\Delta t$ and must satisfy the standard NSFD consistency conditions:
$\phi(\Delta t) > 0$, 
$\phi(\Delta t) = \Delta t + O(\Delta t^2)$ as $\Delta t \to0^+$, 
so that \eqref{E2} reduces to the semi-implicit Euler discretization of \eqref{E1} in the small-step limit, and it is a first-order consistent scheme. 
The simplest admissible choice is $\phi(\Delta t) = \Delta t$ itself; more elaborate choices, e.g.\ $\phi(\Delta t) = (1-e^{-\lambda \Delta t})/\lambda$, can be used to match a known local rate $\lambda$ exactly, as discussed in \cite{mickens1994nonstandard, Zinihi2025NSFD, Zinihi2026NSFDpLap}.

Since the loss term is linear in $x_i^{n+1}$, \eqref{E2} can be solved for $x_i^{n+1}$ in closed form:
\begin{equation}\label{E3}
   x_i^{n+1} = \frac{x_i^n + \phi(\Delta t)\, G_{\theta,i}(x^n,t_n)}
   {1 + \phi(\Delta t)\, L_{\theta,i}(x^n,t_n)}, \quad i = 1,\dots,d.
\end{equation}
Unlike a generic implicit scheme, \eqref{E3} requires no nonlinear solving at each step: it is an explicit, elementwise rational function of the network outputs at $x^n$.
Consequently, \eqref{E3} defines a differentiable computational layer that can directly replace a conventional numerical integration layer in NODE training.
Since the update consists only of differentiable operations with a strictly positive denominator, it is fully compatible with standard automatic differentiation frameworks.

\begin{proposition}[Consistency]\label{P1}
Assume that the exact solution $x(t)$ of \eqref{E1} is sufficiently smooth on $[0,T]$, and that $G_i$ and $L_i$ are continuously differentiable.
Then the NSFD scheme \eqref{E3} is first-order consistent in time.
\end{proposition}
\begin{proof}
Let $x(t)$ be the exact solution of \eqref{E1}. 
By Taylor expansion, we have that
\begin{equation*}
   x_i(t_{n+1})=x_i(t_n)+\Delta t\,\dot{x}_i(t_n)+O(\Delta t^2).
\end{equation*}
Using \eqref{E1}, this becomes
\begin{equation*}
   x_i(t_{n+1})=x_i(t_n) +\Delta t\Big(G_i(x(t_n),t_n)-L_i(x(t_n),t_n)x_i(t_n)\Big)
  +O(\Delta t^2).
\end{equation*}
Now evaluate the NSFD update \eqref{E2} along the exact solution. 
Since
    $x_i(t_{n+1})=x_i(t_n)+O(\Delta t)$,
and $G_i$, $L_i$ are smooth, we have
\begin{equation*}
  L_i\bigl(x(t_n),t_n\bigr)\,x_i(t_{n+1})=
   L_i\bigl(x(t_n),t_n\bigr)\,x_i(t_n) +O(\Delta t).
\end{equation*}
Using $\phi(\Delta t)=\Delta t+O(\Delta t^2)$, 
the discrete update produces the same expansion up to $O(\Delta t^2)$. 
Therefore, the local truncation error satisfies
    $\tau_i^n=O(\Delta t^2)$,
which proves first-order consistency in time.
\end{proof}

Thus, the NSFD update \eqref{E2} has the same formal temporal accuracy as the semi-implicit Euler method and additionally preserves positivity unconditionally.

\begin{theorem}[Unconditional Positivity and Boundedness]\label{T1}
Let $x^0 \in \mathbb{R}_+^d$. 
Then, the sequence generated by \eqref{E3} satisfies $x^n\in\mathbb{R}_+^d$ for all $n \ge 0$.
Furthermore, if the scheme preserves the invariant set $\Omega \subset \mathbb{R}_+^d$, then $x^n \in \Omega$ for all $n \ge 0$, and thus the numerical solution is bounded.
\end{theorem}
\begin{proof}
By induction on $n$. 
The base case $x^0 \in \mathbb{R}_+^d$ holds by assumption. 
Suppose $x^n \in \mathbb{R}_+^d$. 
Then for each $i$, the numerator of \eqref{E3} is a sum of two nonnegative terms, $x_i^n \ge 0$ and $\phi(\Delta t) \,G_{\theta,i}(x^n,t_n) \ge 0$, and the denominator is $1 + \phi(\Delta t)\, L_{\theta,i}(x^n,t_n) \ge 1 > 0$.
Hence $x_i^{n+1} \ge 0$ for every $i$, so $x^{n+1} \in \mathbb{R}_+^d$.
If the scheme also preserves a bounded invariant region $\Omega$, then $x^n\in\Omega$ for all $n\ge 0$. Since $\Omega$ is bounded, the numerical solution is bounded.
\end{proof}

Note that no bound on the step size $\Delta t$ is required. 
The result holds for arbitrarily large step sizes, which is the defining feature of an NSFD scheme as opposed to a standard scheme whose positivity is only conditional on a step-size restriction.



\begin{proposition}[Differentiability]\label{P2}
If $g_{\theta,i}, \ell_{\theta,i}$ are differentiable with respect to $\theta$ for
every $x,t$, then $x_i^{n+1}$ is also differentiable with respect to $\theta$ for every $n$.
\end{proposition}
\begin{proof}
Immediate: \eqref{E3} composes the differentiable softplus map, elementary arithmetic operations with a strictly positive denominator, and $g_{\theta,i},\ell_{\theta,i}$, all of these are differentiable, and the composition is applied recursively in $n$.
\end{proof}

\begin{remark}
Proposition~\ref{P2} allows \eqref{E3} to be used as a drop-in replacement for a standard NODE solver step in existing training pipelines. Training proceeds exactly as it would for a discretized NODE,
with no need for an adjoint sensitivity equation or an implicit-function-theorem solve, since \eqref{E3} is already explicit in $x^{n+1}$.
\end{remark}

    
\section{Towards Exact Conservation}\label{S4}
Many epidemic models additionally conserve a total quantity, e.g.\ a closed population $N = S+I+R$. 
However, the gain/loss parameterization of Section~\ref{S2} does not enforce this by itself, because $G_{\theta,i}$ and $L_{\theta,i}$ are independent across $i$. 
Exact conservation requires that each unit of loss from one compartment appear as a matching unit of gain in another. This structure is studied under the name of \textit{production-destruction systems}, for which positivity and conservation-preserving schemes 
are known \cite{Burchard2003}. 
We outline the corresponding extension here as a natural next step rather than a result we prove in full. We replace the independent networks $G_{\theta,i}$, $L_{\theta,i}$ by a single set of pairwise, nonnegative flux networks $\Phi_{\theta,ij}(x,t) \ge 0$ ($i \neq j$) representing the learned rate of transfer from compartment $i$ to compartment $j$:
\begin{equation*}
    G_{\theta,i} = \sum_{j \neq i} \Phi_{\theta,ji} \quad  \text{ and } \quad
    L_{\theta,i}\, x_i = \sum_{j \neq i} \Phi_{\theta,ij}.
\end{equation*}
At the continuous level this decomposition telescopes to exactly 
\begin{equation*}
    \sum_i \dot{x}_i = \sum_i \sum_{j \neq i} \Phi_{\theta,ji} - \sum_i \sum_{j \neq i} \Phi_{\theta,ij} = 0,
\end{equation*} 
so the total population is conserved. 
Simultaneously achieving the \textit{discrete} analogue with unconditional positivity is precisely the problem  modified Patankar Runge-Kutta (MPRK) schemes solve for prescribed production-destruction systems, by treating the full off-diagonal flux structure implicitly and solving the resulting (sparse, nonnegative) linear system at each step rather than the diagonal closed form \eqref{E3}. 
A full positivity-and-conservation proof for the learned flux network $\Phi_\theta$, together with the resulting linear-solve training step, is left for future work. We mention this here because it directly targets the stronger "$N$ exactly constant" property that motivated this project, beyond the positivity guarantee already established in Theorem~\ref{T1}.

\section{Example and Numerical Results}\label{S5}
We illustrate the proposed method on the classical SIR model
\begin{equation*}
   \dot{S} = -\beta S I, \quad 
   \dot{I} = \beta SI - \gamma I, \quad 
   \dot{R} = \gamma I,
\end{equation*}
which is already in gain/loss form \eqref{E1} with 
\begin{equation*}
     G_S = 0, \quad L_S = \beta I, \quad G_I = \beta S I, \quad 
     L_I = \gamma, \quad G_R = \gamma I, \quad L_R = 0. 
\end{equation*}
We treat the parameters $\beta$, $\gamma$ as unknown and replace them with the gain/loss networks of Section~\ref{S2}. 
The resulting NSFD update, from \eqref{E3}, is
\begin{equation*}
   S^{n+1} = \frac{S^n}{1 + \phi(\Delta t)\, L_{\theta,S}(x^n,t_n)}, \quad
   I^{n+1} = \frac{I^n + \phi(\Delta t)\, G_{\theta,I}(x^n,t_n)}
    {1 + \phi(\Delta t)\, L_{\theta,I}(x^n,t_n)},
\end{equation*}
and
\begin{equation*}
R^{n+1} = R^n + \phi(\Delta t)\, G_{\theta,R}(x^n,t_n). 
\end{equation*}

\begin{remark}[A free conservation law for the closed SIR model]\label{R1}
Because $\dot{S} + \dot{I} + \dot{R} \equiv 0$, the total population $N=S+I+R$ is conserved by the true dynamics for any $\beta$, $\gamma$, and $R(t) = N_0 - S(t) - I(t)$. 
This allows us to avoid the general flux construction of Section~\ref{S4} here. 
We only model the $(S,I)$ system with the gain/loss networks $L_{\theta,S}$, $G_{\theta,I}$, $L_{\theta,I}$, as well as the update equation~\eqref{E3}, and we recover $R^n := N_0-S^n-I^n$ algebraically. 
Section~\ref{S4}'s construction remains necessary when no such algebraic reduction exists (e.g. open populations, or conservation laws not equal to a simple state sum).
One caveat: nothing here guarantees $S+I\le N_0$, so the recovered $R^n$ is not itself covered by Theorem~\ref{T1}, which applies modeled states $(S,I)$; we report this quantity honestly in this section rather than assume it away.
\end{remark}

\textbf{Setup.} We generated a reference trajectory from the SIR model above with $\beta = 0.4$, $\gamma = 0.1$, $S(0)=0.99$, $I(0)=0.01$, $R(0)=0$ (so $N_0=1$), integrated it over $t \in [0,100]$.
We then train two models on the same noisy, subsampled observations of $(S,I)$ at $\Delta t=1$ (holding out the tail of the trajectory for extrapolation testing):

\begin{itemize}
\item \textit{Baseline (B-NODE)}: an unconstrained NODE on the full 3D state $(S,I,R)$, $\dot{x} = f_\theta(x,t)$ with $f_\theta$ a small Multilayer Perceptron (MLP), integrated 
with explicit Euler and trained by backpropagation through the unrolled
steps.
\item \textit{NSFD-NODE}: the reduced $(S,I)$ gain/loss network of Remark~\ref{R1} with the update \eqref{E3}, trained in the same manner (Proposition~\ref{P2}).
\end{itemize}
Both models use MLPs of the same size (two hidden layers, 32 units, $\tanh$ activations, gradient norm clipped to 5), so any difference in behavior is due to the discretization and output parameterization rather than model capacity. 
We use the same architecture-matched comparison logic to isolate the effect of structuring a NODE's vector field in \cite{buzhardt2025relationship}.
Both models converge to a comparably low training loss ($\sim 10^{-4}$), indicating that the differences in the predictions stem from the discretization strategy rather than differences in fitting accuracy.

\textbf{Metrics.} We evaluated the two models, which were trained for 3000 epochs, without retraining, at three step sizes: a finer-than-training step size ($\Delta t=0.5$), the training step size ($\Delta t=1$), and a \textit{coarse-step stress test} ($\Delta t=5$, five times coarser than the training step). 
This directly targets Theorem~\ref{T1}, whose guarantee holds for every $\Delta t$, not only the one used in training. 
We report: 
$(i)$ the minimum state value and the number of negative entries, over the full $(S,I,R)$ trajectory; 
$(ii)$ the root-mean-square error (RMSE) on $(S,I)$ over the held-out extrapolation window; and 
$(iii)$ the drift $\max_n |N^n - N_0|$ in total population. 

\textbf{Results.} Tables~\ref{Tab1} and \ref{Tab2} report the outcome, and Figure~\ref{Fig1} shows the trajectories at $\Delta t=1$. 
Two things stand out. 
First, $N$-conservation is exact for NSFD-NODE at every step size by construction (drift $\approx 0$ to floating-point precision) by construction, while the unconstrained baseline drifts substantially ($2.9$ at $\Delta t \le 1$, and $18.35$ under the coarse-step test) despite
matching training loss. 
Second, positivity: NSFD-NODE's modeled states $(S,I)$ never go negative, as guaranteed.
The baseline substantially violates positivity  at every step size, and catastrophically so under the coarse-step test (minimum $-3.19\,e^{-00}$, RMSE $3.32$), despite comparably low training loss. Evidently, low training error does not guarantee of physically consistent unrolled dynamics.

\begin{table}[h]
\centering
\setlength{\tabcolsep}{0.3cm}
\caption{Numerical comparison of the unconstrained B-NODE and NSFD-NODE, 
at a finer-than-training step $\Delta t=0.5$, the training step size $\Delta t=1$ and under a coarse-step stress test $\Delta t=5$.}\label{Tab1}
\resizebox{\textwidth}{!}{
\begin{tabular}{cccccc}
\toprule
Step Size & Model & Min.\ State Value & Negative Entries & RMSE & Max. $|N^n - N^0|$ \\
\midrule\midrule
\multirow{2}{*}{$\Delta t=0.5$} & B-NODE & $-2.4504\,e^{-02}$ & $72$ & $0.0171$ & $2.9403$ \\
& NSFD-NODE & $\phantom{-}0.0000$ & $0$ & $0.0146$ & $0.0000$ \\
\midrule
\multirow{2}{*}{$\Delta t=1$} & B-NODE & $-2.4712\,e^{-02}$ & $37$ & $0.0174$ & $2.9409$ \\
& NSFD-NODE& $\phantom{-}0.0000$ & $0$ & $0.0145$ & $0.0000$ \\
\midrule
\multirow{2}{*}{$\Delta t=5$} & B-NODE & $-3.1906\,e^{-00}$ & $19$ & $3.3236$ & $18.3525$ \\
& NSFD-NODE & $\phantom{-}0.0000$ & $0$ & $0.0249$ & $0.0000$ \\
\bottomrule
\end{tabular}}
\end{table}

\begin{table}[h]
\centering
\setlength{\tabcolsep}{0.25cm}
\caption{Minimum and maximum values of the state variables obtained by the B-NODE and the proposed NSFD-NODE for different time step sizes.}\label{Tab2}
\resizebox{\textwidth}{!}{
\begin{tabular}{ccccccccccccc}
\toprule
\multirow{2}{*}{Step Size}
& \multicolumn{2}{c}{$S$ B-NODE}
& \multicolumn{2}{c}{$S$ NSFD-NODE}
& \multicolumn{2}{c}{$I$ B-NODE}
& \multicolumn{2}{c}{$I$ NSFD-NODE}
& \multicolumn{2}{c}{$R$ B-NODE}
& \multicolumn{2}{c}{$R$ NSFD-NODE} \\
\cmidrule(lr){2-3}\cmidrule(lr){4-5}
\cmidrule(lr){6-7}\cmidrule(lr){8-9}
\cmidrule(lr){10-11}\cmidrule(lr){12-13}
& Min. & Max. & Min. & Max.
& Min. & Max. & Min. & Max.
& Min. & Max. & Min. & Max. \\
\midrule\midrule
$\Delta t=0.5$
& 0.0043 & 1.0198 & 0.0003 & 0.9900
& -0.0245 & 0.4048 & 0.0100 & 0.3913
& 0.0000 & 3.9605 & 0.0010 & 0.9868 \\

$\Delta t=1$
& 0.0043 & 1.0442 & 0.0004 & 0.9900
& -0.0247 & 0.4080 & 0.0100 & 0.3993
& 0.0000 & 3.9613 & 0.0007 & 0.9867 \\

$\Delta t=5$
& -3.1906 & 6.6857 & 0.0011 & 0.9900
& -1.3762 & 5.8326 & 0.0100 & 0.4443
& -0.7115 & 11.6093 & 0.0000 & 0.9852 \\
\bottomrule
\end{tabular}}
\end{table}

\begin{figure}[H]
\centering
\includegraphics[width=\textwidth]{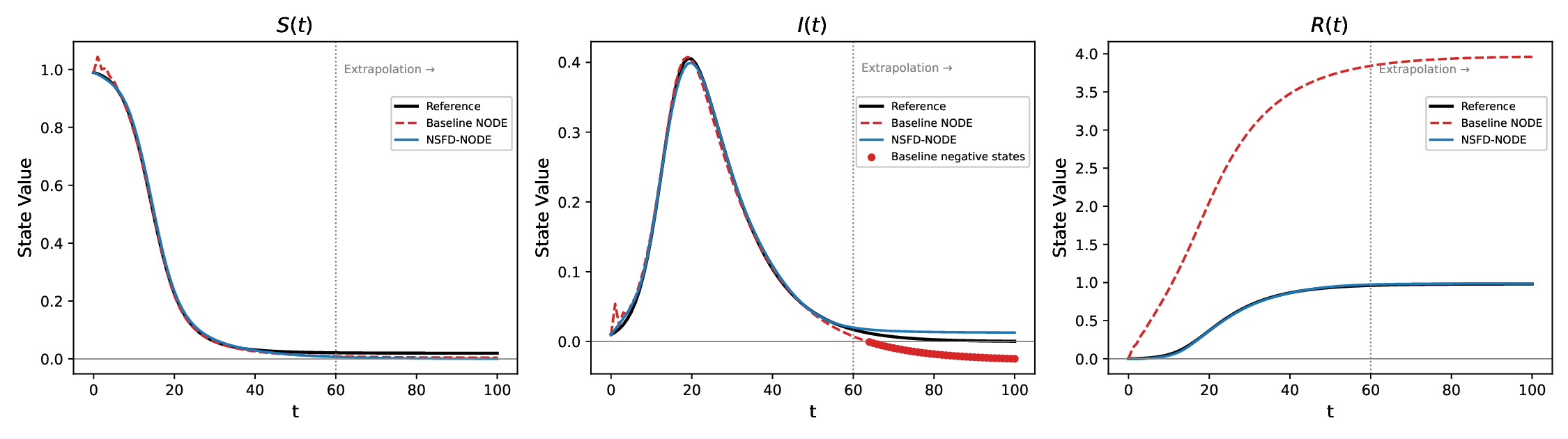}
\caption{Reference, B-NODE (negative excursions marked), and NSFD-NODE trajectories for $S(t)$, $I(t)$, and $R(t)$ at $\Delta t = 1$. 
$R$ is recovered algebraically for NSFD-NODE (Remark~\ref{R1}).
Dotted line: end of training window.}\label{Fig1}
\end{figure}

The main advantage of the scheme is that positivity is a structural consequence of equation \eqref{E3} rather than a property encouraged by a penalty term.
Theorems~\ref{T1} holds for every $\theta$ and every $\Delta t$, with no tuning of a penalty weight and no risk of violating the constraint on out-of-distribution inputs. This is precisely where penalty-based approaches tend to fail. 
However, the cost is architectural; the network must be written in gain/loss form, which is natural for compartmental and production-destruction systems, but restrictive compared to an arbitrary $f_\theta$.

\section{Conclusion}\label{S6}
We introduced an NSFD discretization for neural ODEs based on a gain/loss parameterization. 
This update preserves positivity unconditionally, remains bounded, and is differentiable in closed form. These properties make it directly compatible with standard Neural ODE training pipelines. 
This method addresses a critical limitation of scientific neural ODEs by enforcing physically meaningful nonnegative states via structural discretization instead of a penalty or projection step.

Our work establishes a rigorous connection between structure-preserving numerical analysis and neural ODEs. It demonstrates that qualitative properties can be enforced by design rather than through penalties or projections.



\section*{Declarations}

\subsection*{Conflict of Interest} 
\noindent
The authors declared that they have no conflict of interest.

\subsection*{Data Availability} 
\noindent
No data was used for the research described in the article. 

%




\bibliographystyle{unsrt}
\bibliography{References}


\end{document}